\documentclass[12pt,reqno]{amsart}
\usepackage{etex,mathtools}
\usepackage{mathdots,longtable}
\usepackage{amscd,amssymb,amsmath,multicol,tikz,float,mathrsfs}
\usepackage[left=1in,right=1in]{geometry}
\begin{document}

\restylefloat{table}
\newtheorem{thm}[equation]{Theorem}
\numberwithin{equation}{section}
\newtheorem{cor}[equation]{Corollary}
\newtheorem{expl}[equation]{Example}
\newtheorem{rmk}[equation]{Remark}
\newtheorem{conv}[equation]{Convention}
\newtheorem{claim}[equation]{Claim}
\newtheorem{lem}[equation]{Lemma}
\newtheorem{sublem}[equation]{Sublemma}
\newtheorem{conj}[equation]{Conjecture}
\newtheorem{defin}[equation]{Definition}
\newtheorem{diag}[equation]{Diagram}
\newtheorem{prop}[equation]{Proposition}
\newtheorem{notation}[equation]{Notation}
\newtheorem{tab}[equation]{Table}
\newtheorem{fig}[equation]{Figure}
\newcounter{bean}
\renewcommand{\theequation}{\thesection.\arabic{equation}}

\raggedbottom \voffset=-.7truein \hoffset=0truein \vsize=8truein
\hsize=6truein \textheight=8truein \textwidth=6truein
\baselineskip=18truept
\def\mapleft#1{\smash{\mathop{\longleftarrow}\limits^{#1}}}
\def\mapright#1{\ \smash{\mathop{\longrightarrow}\limits^{#1}}\ }
\def\ml#1{\,\smash{\mathop{\leftarrow}\limits^{#1}}\,}
\def\mapup#1{\Big\uparrow\rlap{$\vcenter {\hbox {$#1$}}$}}
\def\mapdown#1{\Big\downarrow\rlap{$\vcenter {\hbox {$\ssize{#1}$}}$}}
\def\mapne#1{\nearrow\rlap{$\vcenter {\hbox {$#1$}}$}}
\def\mapse#1{\searrow\rlap{$\vcenter {\hbox {$\ssize{#1}$}}$}}
\def\mapr#1{\smash{\mathop{\rightarrow}\limits^{#1}}}
\def\Mt{\widetilde{\M}}
\def\ss{\smallskip}
\def\s{\sigma}
\def\Bt{\widetilde{\mathcal{B}}}
\def\l{\lambda}
\def\Ah{\widehat{A}}
\def\Bh{\widehat{B}}
\def\vp{v_1^{-1}\pi}
\def\at{{\widetilde\alpha}}
\def\At{\widetilde{\mathcal{A}}}
\def\as{\mathscr{A}}
\def\Ast{\widetilde{\as}}
\def\Mct{\widetilde{\mathcal{M}}}
\def\sm{\wedge}
\def\la{\langle}
\def\ra{\rangle}
\def\lar{\leftarrow}
\def\ev{\text{ev}}
\def\od{\text{od}}
\def\on{\operatorname}
\def\ol#1{\overline{#1}{}}
\def\spin{\on{Spin}}
\def\cat{\on{cat}}
\def\Lbar{\overline{\Lambda}}
\def\qed{\quad\rule{8pt}{8pt}\bigskip}
\def\ssize{\scriptstyle}
\def\a{\alpha}
\def\bz{{\Bbb Z}}
\def\Rhat{\hat{R}}
\def\im{\on{im}}
\def\ct{\widetilde{C}}
\def\ext{\on{Ext}}
\def\sq{\on{Sq}}
\def\eps{\epsilon}
\def\ar#1{\stackrel {#1}{\rightarrow}}
\def\br{{\bold R}}
\def\bC{{\bold C}}
\def\bA{{\bold A}}
\def\bB{{\bold B}}
\def\bD{{\bold D}}
\def\bC{{\bold C}}
\def\bh{{\bold H}}
\def\bQ{{\bold Q}}
\def\bP{{\bold P}}
\def\bx{{\bold x}}
\def\bo{{\bold{bo}}}
\def\dh{\widehat{d}}
\def\A{\mathcal{A}}
\def\B{\mathcal{B}}
\def\si{\sigma}
\def\Vbar{{\overline V}}
\def\dbar{{\overline d}}
\def\wbar{{\overline w}}
\def\Sum{\sum}
\def\tfrac{\textstyle\frac}

\def\tb{\textstyle\binom}
\def\Si{\Sigma}
\def\w{\wedge}
\def\equ{\begin{equation}}
\def\b{\beta}
\def\G{\Gamma}
\def\L{\Lambda}
\def\g{\gamma}
\def\d{\delta}
\def\k{\kappa}
\def\psit{\widetilde{\Psi}}
\def\tht{\widetilde{\Theta}}
\def\psiu{{\underline{\Psi}}}
\def\thu{{\underline{\Theta}}}
\def\aee{A_{\text{ee}}}
\def\aeo{A_{\text{eo}}}
\def\aoo{A_{\text{oo}}}
\def\aoe{A_{\text{oe}}}
\def\vbar{{\overline v}}
\def\endeq{\end{equation}}
\def\xhat{\widehat{x}}
\def\sn{S^{2n+1}}
\def\zp{\bold Z_p}
\def\cR{{\mathcal R}}
\def\P{{\mathcal P}}
\def\cQ{{\mathcal Q}}
\def\cj{{\cal J}}
\def\zt{{\bold Z}_2}
\def\bs{{\bold s}}
\def\bof{{\bold f}}
\def\bq{{\bold Q}}
\def\be{{\bold e}}
\def\Hom{\on{Hom}}
\def\ker{\on{ker}}
\def\kot{\widetilde{KO}}
\def\coker{\on{coker}}
\def\da{\downarrow}
\def\colim{\operatornamewithlimits{colim}}
\def\zphat{\bz_2^\wedge}
\def\io{\iota}
\def\om{\omega}
\def\Prod{\prod}
\def\e{{\cal E}}
\def\zlt{\Z_{(2)}}
\def\exp{\on{exp}}
\def\abar{{\overline a}}
\def\xbar{{\overline x}}
\def\ybar{{\overline y}}
\def\zbar{{\overline z}}
\def\mbar{{\overline m}}
\def\nbar{{\overline n}}
\def\sbar{{\overline s}}
\def\kbar{{\overline k}}
\def\bbar{{\overline b}}
\def\et{{\widetilde E}}
\def\ni{\noindent}
\def\tsum{\textstyle \sum}
\def\coef{\on{coef}}
\def\den{\on{den}}
\def\lcm{\on{l.c.m.}}
\def\Ext{\operatorname{Ext}}
\def\iso{\approx}
\def\lra{\longrightarrow}
\def\vi{v_1^{-1}}
\def\ot{\otimes}
\def\psibar{{\overline\psi}}
\def\thbar{{\overline\theta}}
\def\Mh{{\widehat M}}
\def\exc{\on{exc}}
\def\ms{\medskip}
\def\ehat{{\hat e}}
\def\etao{{\eta_{\text{od}}}}
\def\etae{{\eta_{\text{ev}}}}
\def\dirlim{\operatornamewithlimits{dirlim}}
\def\gt{\widetilde{L}}
\def\lt{\widetilde{\lambda}}
\def\st{\widetilde{s}}
\def\ft{\widetilde{f}}
\def\sgd{\on{sgd}}
\def\lfl{\lfloor}
\def\rfl{\rfloor}
\def\ord{\on{ord}}
\def\gd{{\on{gd}}}
\def\rk{{{\on{rk}}_2}}
\def\nbar{{\overline{n}}}
\def\MC{\on{MC}}
\def\lg{{\on{lg}}}
\def\cH{\mathcal{H}}
\def\cS{\mathcal{S}}
\def\cP{\mathcal{P}}
\def\N{{\Bbb N}}
\def\Z{{\Bbb Z}}
\def\Q{{\Bbb Q}}
\def\R{{\Bbb R}}
\def\C{{\Bbb C}}
\def\Lb{\overline\Lambda}
\def\mo{\on{mod}}
\def\xt{\times}
\def\notimm{\not\subseteq}
\def\Remark{\noindent{\it  Remark}}
\def\kut{\widetilde{KU}}
\def\Eb{\overline E}
\def\*#1{\mathbf{#1}}
\def\0{$\*0$}
\def\1{$\*1$}
\def\22{$(\*2,\*2)$}
\def\33{$(\*3,\*3)$}
\def\ss{\smallskip}
\def\ssum{\sum\limits}
\def\dsum{\displaystyle\sum}
\def\la{\langle}
\def\ra{\rangle}
\def\on{\operatorname}
\def\proj{\on{proj}}
\def\od{\text{od}}
\def\ev{\text{ev}}
\def\o{\on{o}}
\def\U{\on{U}}
\def\lg{\on{lg}}
\def\a{\alpha}
\def\bz{{\Bbb Z}}
\def\ccM{{\Bbb M}}
\def\E{\mathcal{E}}
\def\eps{\varepsilon}
\def\bc{{\bold C}}
\def\bN{{\bold N}}
\def\bB{{\bold B}}
\def\bW{{\bold W}}
\def\nut{\widetilde{\nu}}
\def\tfrac{\textstyle\frac}
\def\b{\beta}
\def\G{\Gamma}
\def\g{\gamma}
\def\zt{{\Bbb Z}_2}
\def\zth{{\bold Z}_2^\wedge}
\def\bs{{\bold s}}
\def\bx{{\bold x}}
\def\bof{{\bold f}}
\def\bq{{\bold Q}}
\def\be{{\bold e}}
\def\lline{\rule{.6in}{.6pt}}
\def\xb{{\overline x}}
\def\xbar{{\overline x}}
\def\ybar{{\overline y}}
\def\zbar{{\overline z}}
\def\ebar{{\overline e}}
\def\nbar{{\overline n}}
\def\ubar{{\overline u}}
\def\bbar{{\overline b}}
\def\et{{\widetilde e}}
\def\M{\mathcal{M}}
\def\lf{\lfloor}
\def\rf{\rfloor}
\def\ni{\noindent}
\def\ms{\medskip}
\def\Dhat{{\widehat D}}
\def\what{{\widehat w}}
\def\Yhat{{\widehat Y}}
\def\abar{{\overline{a}}}
\def\minp{\min\nolimits'}
\def\sb{{$\ssize\bullet$}}
\def\mul{\on{mul}}
\def\N{{\Bbb N}}
\def\Z{{\Bbb Z}}
\def\Q{{\Bbb Q}}
\def\R{{\Bbb R}}
\def\C{{\Bbb C}}
\def\Xb{\overline{X}}
\def\eb{\overline{e}}
\def\notint{\cancel\cap}
\def\cS{\mathcal S}
\def\cR{\mathcal R}
\def\el{\ell}
\def\TC{\on{TC}}
\def\GC{\on{GC}}
\def\wgt{\on{wgt}}
\def\Ht{\widetilde{H}}
\def\wbar{\overline w}
\def\dstyle{\displaystyle}
\def\Sq{\on{sq}}
\def\Om{\Omega}
\def\ds{\dstyle}
\def\tz{tikzpicture}
\def\zcl{\on{zcl}}
\def\bd{\bold{d}}
\def\cM{\mathcal{M}}
\def\io{\iota}
\def\od{\operatorname{od}}
\def\odprod{\on{odpr}}
\def\uns{\on{uns}}
\def\stab{\on{stab}}
\def\Vb#1{{\overline{V_{#1}}}}
\def\Ebar{\overline{E}}
\def\lb{\,\begin{picture}(-1,1)(1,-1)\circle*{3.5}\end{picture}\ }
\def\rlb{\,\begin{picture}(-1,1)(1,-1) \circle*{4.5}\end{picture}\ }
\def\lbb{\,\begin{picture}(-1,1)(1,-1)\circle*{8}\end{picture}\ }
\def\zp{\Z_p}
\def\lbr{\,\begin{picture}(-1,1)(1,-1)[dashed]\circle*{3.5}\end{picture}\ }
\def\llb{\,\begin{picture}(-1,1)(1,-1)\circle*{2.6}\end{picture}\ }
\def\blb{\,\begin{picture}(-1,1)(1,-1) \circle*{5.8}\end{picture}\ }
\def\alh{\widehat{\a}}
\def\sihat{\widehat{\sigma}}
\setcounter{MaxMatrixCols}{15}
\title
{Some 2-adic integers related to the odd part of $2^e!$}
\author{Donald M. Davis}
\address{Department of Mathematics, Lehigh University\\Bethlehem, PA 18015, USA}
\email{dmd1@lehigh.edu}
\date{October 24, 2025}
\begin{abstract} The odd part of $2^e!$ as $e\to\infty$ leads to a 2-adic integer $z$. The bits of $z$ were publicized in OEIS-A359349, where two conjectures were made, relevant to computing $z$. We prove both of those conjectures. A second 2-adic integer, the limit of $((2^e-1)!!-1)/2^e$, plays a key role in one proof.\end{abstract}
\maketitle
\section{Introduction}\label{intro}
In \cite{MAA}, the author noted that the odd part of $2^e!$ and of $2^{e-1}!$ agree mod $2^e$, and so the 2-adic limit as $e$ approaches $\infty$ is a 2-adic integer, which we will call $z$. In OEIS-A359349(\cite{OEIS}), the author and Jon E. Schoenfield publicized the sequence of bits of $z$ and made two conjectures. One involved a relationship between the bits of $z$ and some of the unstable bits in the odd part of $2^e!$, while the other leads to a more efficient way of computing $z$. In this paper we prove both conjectures and some generalizations.

In this introductory section, we review the two conjectures, stating them as theorems.
In Sections \ref{sec2} and \ref{sec3}, we prove generalizations of both.

Let $\nu(n)$ denote the exponent of 2 in the prime factorization of $n$, and $\od(n)=2/2^{\nu(n)}$ the odd part of $n$. Then $\od(2^e!)=2^e!/2^{2^e-1}$. In Figure \ref{table} we tabulate the first 40 bits in the backward binary expansion (BBE) of $\od(2^e!)$ for $2\le e\le30$. In \cite{OEIS}, a larger table (64 bits for $e\le40$) was presented.

\begin{table}[h]
\caption{first 32 bits in BBE of $\od(2^e!)$}
\label{table}
\begin{tabular}{r|l}
2&110$\ $0000000000000000000000000000000000000\\
3&1101$\ $110010000000000000000000000000000000\\
4&11010$\ $11101110111011100000110010000000000\\
5&110100$\ $1011001110100011000001101010001001\\
6&1101000$\ $000000001101000110100010110011110\\
7&11010001$\ $10010101001010001010001101011100\\
8&110100010$\ $0111100111001000110000010011101\\
9&1101000101$\ $000110101110110000001011000011\\
10&11010001011$\ $10101101100010111001010101001\\
11&110100010110$\ $0010011111110101111010110111\\
12&1101000101101$\ $110111110000001100111011011\\
13&11010001011010$\ $11100001101100011101011000\\
14&110100010110100$\ $1011001111011110110100001\\
15&1101000101101000$\ $000011110100000100110000\\
16&11010001011010001$\ $10000001011001001111011\\
17&110100010110100010$\ $0101001000011101001001\\
18&1101000101101000101$\ $011101001000000000010\\
19&11010001011010001011$\ $00000101101001010101\\
20&110100010110100010111$\ $1001101111101111101\\
21&1101000101101000101110$\ $011011110000010111\\
22&11010001011010001011101$\ $00100001101011100\\
23&110100010110100010111011$\ $1101001111111011\\
24&1101000101101000101110110$\ $111101110010100\\
25&11010001011010001011101100$\ $10000001111101\\
26&110100010110100010111011000$\ $0101001100111\\
27&1101000101101000101110110001$\ $011101101100\\
28&11010001011010001011101100011$\ $00000011011\\
29&110100010110100010111011000111$\ $1001011000\\
30&1101000101101000101110110001110$\ $011111001
\end{tabular}
\end{table}

Bits $0$ through $e$ of $\od(2^e!)$ are {\it stable}; they agree with those of $z$. The {\it unstable} bits of $\od(2^e!)$ are those in position $\ge e+1$; they appear to the right of  the space in Figure \ref{table}. Note that, for $e>d$, the first $d$ unstable bits of $\od(2^e!)$ occur in the same positions as the last $d$ stable bits of $\od(2^{e+d}!)$. The latter are the stable bits of $z$ in position $e+1$ through $e+d$. Let $\uns(e,d)$ and $\stab(e+1,d)$ denote the numbers whose BBE's are these sequences of $d$ bits. The following theorem was the first conjecture of \cite{OEIS}.

\begin{thm}\label{thm1} There is a 2-adic integer $K$ such that, for all $d$ and $e>d$
$$\uns(e,d)+K\equiv \stab(e+1,d).$$
\end{thm}
\begin{expl} The BBE of $K$ begins $1011011$. The BBE's of $\uns(17,7)$ and $\stab(18,7)$
are $0101001$ and $1110110$, respectively. After reversing the order of the bits, the theorem is easily verified in this case.
\end{expl}

This relationship between the stable and unstable parts is, at least, a curiosity. It could be useful in calculations. We will see in Section \ref{sec2} that the first $d$ bits of $K$ can be determined from $d$ bits of $z$ and $d$ bits of another 2-adic integer $w$. So, for example, bits 18 through 24 of $z$ can be determined from $\od(2^{17}!)$ and $K$ mod $2^7$, which is an easier calculation than $\od(24!)$.

Let $\odprod(\ell,m)$ denote the product of all odd integers $j$ satisfying $\ell\le j\le m$, and let $$h(m)=\odprod(2^{m-1}+1,2^m-1).$$
We begin with the following elementary proposition.
\begin{prop} For any $e\ge1$,
$$\od(2^e!)=\prod_{m=2}^eh(m)^{e+1-m}.$$
\end{prop}
\begin{proof} Each factor $j$ of $h(m)$ occurs with coefficient $2^i$ for $0\le i\le e-m$ in $2^e!$.
\end{proof}
This yields a method of computing $\od(2^e!)$ mod $2^B$, reducing mod $2^B$ at each step. The following theorem, which was the second conjecture of \cite{OEIS}, makes it more efficient.
\begin{thm}\label{thm2} If $2\le m-1\le B\le 3m-7$, and $d=2+\lfloor \frac{B-m}2\rfloor$, then
$$h(m)\equiv \odprod(2^{m-1}+1,2^{m-1}+2^d-1)^{2^{m-1-d}}\pmod{2^B}.$$
\end{thm}
The advantage is that now $h(m)$ requires $2^{d-1}+m-2-d$ multiplications (always reducing mod $2^B$) compared with $2^{m-2}-1$ multiplications.

\section{A formula for the 2-adic integer $K$ of Theorem \ref{thm1}}\label{sec2}
In this section, we prove Theorem \ref{thm1} and give a formula for the 2-adic integer $K$ that occurs in it.
We begin by reviewing the proof in \cite{MAA} of existence of the 2-adic integer $z$, as some of the ingredients will be useful later.

\begin{lem} \label{bij} Let $I_e=\{i:2^{e-1}<i\le 2^e\}$ and $S_e=\{j: j\text{ odd and }1\le j<2^e\}$. Then $\od:I_e\to S_e$ is bijective.\end{lem}
\begin{proof} The inverse function $\phi$ is defined by $\phi(u)=2^tu$ where $t=\max\{k:2^ku\le 2^e\}$.\end{proof}
\begin{lem} \label{Gauss} If $e\ge3$, the product of all odd positive integers less than $2^e$ is $\equiv2^e+1\pmod{2^{e+1}}$.\end{lem}
\begin{proof} We begin with the proof from \cite[Lemma 1]{Gran} that the product is 1 mod $2^e$. Pair each element with its inverse in $\Z/2^e$. Only $\pm1$ and $2^{e-1}\pm1$ equal their own inverse, and their product is 1.

Let $P$ be the set of pairs $(a,b)$ with $a<b<2^e$ odd and $ab\equiv 2^e+1$ mod $2^{e+1}$. If $(a,b)\in P$, so is $(2^e-b,2^e-a)$ since $a+b$ is even. Moreover, $(a,b)\ne (2^e-b,2^e-a)$ since, if so, then $a(2^e-a)\equiv 2^e+1$ mod $2^{e+1}$, which cannot occur since $a^2\equiv 1$ mod 8. Thus the cardinality of $P$ is even, and the product of all $ab$ with $(a,b)\in P$ is 1 mod $2^{e+1}$. Other pairs $(c,d)$ with $cd\equiv1$ mod $2^e$ have $cd\equiv1$ mod $2^{e+1}$. Finally we have $1$, $2^e-1$, and $2^{e-1}\pm1$, whose product is $\equiv 2^e+1$ mod $2^{e+1}$.
\end{proof}
\begin{cor} For $e\ge3$, $\od(2^{e-1}!)\equiv\od(2^e!)\pmod{2^e}$.\end{cor}
\begin{proof} By Lemmas \ref{bij} and \ref{Gauss},
$$\frac{\od(2^e!)}{\od(2^{e-1}!)}=\prod_{i\in I_e}\od(i)=\prod_{j\in S_e}j\equiv1\pmod{2^e}.$$\end{proof}
\begin{cor}\label{zdef} There is a 2-adic integer $z$ which equals $\od(2^{e-1}!)$ mod $2^e$.\end{cor}

\begin{rmk}{\rm The stronger (mod $2^{e+1}$) part of Lemma \ref{Gauss} was not needed here, but will be used shortly.}\end{rmk}

A stronger version of the next  result will be proved in Theorem \ref{hard}.
\begin{prop} \label{weak1} With $S_e$ as above,
$$\prod_{i\in S_e}i\equiv\prod_{i\in S_e}(2^e+i)\pmod{2^{2e}}.$$\end{prop}
\begin{proof} If $S$ is a set of cardinality $n$, let $\sihat_i(S)=\si_{n-i}(S)$, where $\s$ is the usual elementary symmetric polynomial. Then $\sihat_1(S_e)$ is divisible by $2^e$ since, for odd $j\le 2^{e-1}-1$,
$$\prod_{\substack{i\in S_e\\i\ne j}}i+\prod_{\substack{i\in S_e\\i\ne 2^e-j}}i\text{ is divisible by }2^e.$$
We have
$$\prod_{i\in S_e}(2^e+i)-\prod_{i\in S_e}i=\sum_{j>0}2^{je}\sihat_j(S_e)\equiv 2^{e}\sihat_1(S_e)\equiv0\pmod{2^{2e}}.$$
\end{proof}

Let $(2^e-1)!!=\odprod(1,2^e-1)$.
\begin{cor}\label{wcor} For $e\ge2$, $\dfrac{(2^e-1)!!-1}{2^e}\equiv\dfrac{(2^{e+1}-1)!!-1}{2^{e+1}} \pmod{2^{e-1}}$.\end{cor}
\begin{proof} By Lemma \ref{Gauss}, the two expressions are odd integers. We will show that their ratio is $\equiv1$ mod $2^{e-1}$.
Let $A=(2^e-1)!!-1$. By Lemma \ref{Gauss}, $A=2^eu$ with $u$ odd. By Proposition \ref{weak1}, $\odprod(2^e+1,2^{e+1}-1)=A+1+k2^{2e}$ for some integer $k$. The desired ratio is
\begin{align*}&\frac{(A+1)(A+1+k2^{2e})-1}{2A}\\
=\ &\frac{A^2+2A+(A+1)k2^{2e}}{2A}=2^{e-1}u+1+\frac{(A+1)k2^{2e}}{2^{e+1}u}\equiv1\pmod{2^{e-1}}.\end{align*}
\end{proof}

\begin{defin} We define $w$ to be the 2-adic integer which equals $\dfrac{(2^e-1)!!-1}{2^e}$ mod $2^{e-1}$.\label{wdef}\end{defin}

\noindent The binary expansion of $w$ ends $\cdots 1001110011001$.

We have introduced two 2-adic integers, $z$ and $w$. The next result shows that their product equals the difference of the unstable parts of consecutive rows of Figure \ref{table} in a metastable range.
\begin{thm}\label{prod} The difference of the unstable parts of $\od(2^e!)$ and $\od(2^{e-1}!)$, i.e., $\dfrac{\od(2^e!)-\od(2^{e-1}!)}{2^e}$, is congruent mod $2^{e-1}$ to $w\cdot z$.\end{thm}
\begin{proof} By Lemma \ref{bij}, we have $(2^e-1)!!=\dfrac{\od(2^e!)}{\od(2^{e-1}!)}$.
Thus
$$\frac{(2^e-1)!!-1}{2^e}=\frac{\od(2^e!)-\od(2^{e-1}!)}{2^e\od(2^{e-1}!)},$$
so $$\frac{(2^e-1)!!-1}{2^e}\cdot\od(2^{e-1}!)=\frac{\od(2^e!)-\od(2^{e-1}!)}{2^e}.$$
The result follows now from Corollary \ref{zdef} and Definition \ref{wdef}.
\end{proof}
\begin{expl} The binary expansion of $zw$ ends $\cdots011000010011$. The numbers $\od(2^7!)$ and $\od(2^6!)$ agree mod $2^7$. Beginning in the $2^7$ position, the binary expansion of $\od(2^7!)$ ends $\cdots 1010011$, while that of $\od(2^6!)$ ends with eight 0's. The difference agrees with $zw$ mod $2^6$.\end{expl}

We now state the main theorem of this section.
\begin{thm} \label{Kthm} The 2-adic integer $K$ of Theorem \ref{thm1} equals $-zw$.
\end{thm}
\begin{proof}[Proof of Theorems \ref{thm1} and \ref{Kthm}.] The difference
$\stab(e+1,d)-\uns(e,d)$, as described in the paragraph preceding Theorem \ref{thm1}, equals $\dfrac{\od(2^{e+d}!)-\od(2^e!)}{2^{e+1}}$ mod $2^d$. We have
\begin{align*}&\frac{\od(2^{e+d}!)-\od(2^e!)}{2^{e+1}}\\
=\ &\sum_{i=1}^d\frac{\od(2^{e+i}!)-\od(2^{e+i-1}!)}{2^{e+1}}\\
=\ &\sum_{i=1}^d2^{i-1}\frac{\od(2^{e+i}!)-\od(2^{e+i-1}!)}{2^{e+i}}\\
\equiv\ &\sum_{i=1}^d2^{i-1}zw\pmod{2^e}\\
\equiv\ &\sum_{i=1}^\infty2^{i-1}zw\pmod{2^d}\\
=\ &-zw.\end{align*}
\end{proof}

\begin{expl} The binary expansion of $-zw$ ends $\cdots010111101101$. Add that to the binary number obtained by reversing the order of the first 12 bits after the space on line 14 of Figure \ref{table}, and you obtain the binary number obtained by reversing the order of the last 12 bits before the space on line 26.\end{expl}

\section{Proof of Theorem \ref{thm2}}\label{sec3}
In this section, we prove Theorem \ref{thm2} and some mild generalizations. The bulk of our work is the following strengthening of Proposition \ref{weak1}, the proof of which appears later.
\begin{thm} \label{hard} With $S_e$ as defined in Lemma \ref{bij}, and $A$ any integer,
$$\prod_{i\in S_e}(A2^e+i)\equiv\prod_{i\in S_e}i\pmod{2^{3e-1}}.$$
\end{thm}
\begin{cor} \label{ABcor} For any integers $A$, $B$, and $j$,
$$\prod_{i\in S_e}(A2^e+i)^{2^j}\equiv\prod_{i\in S_e}(B2^e+i)^{2^j}\pmod{2^{3e-1+j}}.$$
\end{cor}
\begin{proof} It is elementary that if $\a\equiv\b$ mod $2t$, then $\a^2\equiv\b^2$ mod
$4t$. We apply this iteratively to Theorem \ref{hard}, and then both expressions in the corollary are congruent to $\prod i^{2^j}$.\end{proof}

\begin{proof}[Proof of Theorem \ref{thm2}.]
We write the conjectured congruences in succession, beginning
\begin{align*}
\odprod(2^{m-1}+1,2^m-1)&\equiv \odprod(2^{m-1}+1,2^{m-1}+2^{m-2}-1)^2\quad (2^{3m-7})\\
\odprod(2^{m-1}+1,2^{m-1}+2^{m-2}-1)^2&\equiv \odprod(2^{m-1}+1,2^{m-1}+2^{m-3}-1)^{2^2}\quad(2^{3m-9})\end{align*}
with arbitrary entry
$$\odprod(2^{m-1}+1,2^{m-1}+2^{d+1}-1)^{2^{m-2-d}}\equiv\odprod(2^{m-1}+1,2^{m-1}+2^d-1)^{2^{m-1-d}}\quad(2^{2d+m-3}).$$
After canceling, this becomes
$$\odprod(2^{m-1}+2^d+1,2^{m-1}+2^{d+1}-1)^{2^{m-2-d}}\equiv\odprod(2^{m-1}+1,2^{m-1}+2^d-1)^{2^{m-2-d}}\quad(2^{2d+m-3}).$$
We can restate this as
$$\prod_{i\in S_d}(2^{m-1}+2^d+i)^{2^{m-2-d}}\equiv\prod_{i\in S_d}(2^{m-1}+i)^{2^{m-2-d}}\pmod{2^{2d+m-3}},$$
and this is a consequence of Corollary \ref{ABcor}.
\end{proof}
We will prove the following two lemmas, from which Theorem \ref{hard} follows easily.
\begin{lem}\label{lem1} $\sihat_1(S_e)\equiv 2^{2e-2}\pmod{2^{2e-1}}$.\end{lem}
\begin{lem}\label{lem2} $\sihat_2(S_e)\equiv2^{e-2}\pmod{2^{e-1}}$.\end{lem}
\begin{proof}[Proof of Theorem \ref{hard}.]
$$\prod_{i\in S_e}(A2^e+i)-\prod_{i\in S_e}i=\sum_{j>0}(A2^e)^j\sihat_j(S_e)\equiv0\pmod{2^{3e-1}}$$
by Lemmas \ref{lem1} and \ref{lem2}, with the argument slightly different for the two parities of $A$.
\end{proof}
\begin{proof}[Proof of Lemma \ref{lem1}.]
$$\sihat_1(S_e)=\sum_{i=1}^{2^{e-2}-1}\biggl(\frac{(2^e-1)!!}{2i+1}+\frac{(2^e-1)!!}{2^e-1-2i}\biggr)=2^e\sum_{i=1}^{2^{e-2}-1}\frac{(2^e-1)!!}{(2i+1)(2^e-1-2i)}.$$
Let $H_e=\dsum_{i=1}^{2^{e-2}-1}\frac{(2^e-1)!!}{(2i+1)(2^e-1-2i)}$. We will prove by induction that $H_e\equiv2^{e-2}$ mod $2^{e-1}$, which implies the lemma.

The claim is true for $e=2$. Assume it true for $e-1$. Mod $2^{e-1}$,
$$H_e\equiv\sum_{i=0}^{2^{e-2}-1}\frac{((2^{e-1}-1)!!)^2}{(2i+1)(2^{e-1}-2i-1)}.$$
The summands for $i$ and $2^{e-2}-1-i$ are equal. Thus $H_e\equiv 2(2^{e-1}-1)!!H_{e-1}$ mod $2^{e-1}$. By the induction hypothesis, we obtain $H_e\equiv2^{e-2}$ mod $2^{e-1}$, as desired.
\end{proof}

\noindent We thank Andrew Granville for providing an alternate proof of Lemma \ref{lem1}.

The following results will be used in the proof of Lemma \ref{lem2}.
\begin{lem}\label{8} Of the $2^{e-1}$ numbers $i^2$ mod $2^e$, $i\in S_e$, there are exactly four having each of the $2^{e-3}$ values less than $2^e$ and $\equiv1$ mod $8$.\end{lem}
\begin{proof} Each of the  $2^{e-3}$ numbers is a quadratic residue, and so must occur as $i^2$ for some $i\in S_e$. It will occur in four ways since
for odd $i<2^{e-1}$, $i$, $2^{e-1}-i$, $i+2^{e-1}$, and $2^e-i$ are distinct numbers with the same square mod $2^e$. Thus the claimed partitioning must hold.
\end{proof}
\begin{lem}\label{9} For $e\ge3$,
$$\sihat_1(1,9,\ldots,2^e-7,1,9,\ldots,2^e-7,1,9,\ldots,2^e-7,1,9,\ldots,2^e-7)\equiv 2^{e-1}\quad (2^e).$$
\end{lem}
\begin{proof}
The proof is by induction. The claim is true for $e=3$ and $4$. [\![$\sihat_1(1,1,1,1)=4$ and $\sihat_1(1,9,1,9,1,9,1,9)=4\cdot9^4+4\cdot9^3=4\cdot9^3\cdot10$.]\!]
For arbitrary $e$, our expression equals $4\cdot9^3\cdots(2^e-7)^3\cdot\sihat_1(1,9,\ldots2^e-7)$. Because of the 4,
we can consider $\sihat_1(1,9,\ldots,2^e-7)$ mod $2^{e-2}$, so we obtain an odd multiple of
$4\cdot\Si$ with
$$\Si=\sihat_1(1,9,\ldots,2^{e-2}-7,1,9,\ldots,2^{e-2}-7,1,9,\ldots,2^{e-2}-7,1,9,\ldots,2^{e-2}-7).$$
By the induction hypothesis, $\Si\equiv2^{e-3}$ mod $2^{e-2}$, and so our desired expression is $\equiv2^{e-1}$ mod $2^e$.

\end{proof}
\begin{prop}\label{sqprop} $\dsum_{i\in S_e}\dfrac{((2^e-1)!!)^2}{i^2}\equiv 2^{e-1}\pmod{2^e}$.\end{prop}
\begin{proof} By Lemma \ref{8},  it equals the expression in Lemma \ref{9}.\end{proof}

\begin{proof}[Proof of Lemma \ref{lem2}.]

 Let $D_e=\{(a,b)\in S_e\times S_e:\ a<b\}$. Note that $\sihat_2(S_e)=\dsum_{(a,b)\in D_e}\dfrac{(2^e-1)!!}{a\cdot b}$, denoted  by $T_e$. Write $T_e=T_{1,e}+T_{2,e}$, where
 $$T_{1,e}=\sum_{\substack{(a,b)\in D_e\\ a\not\equiv b\ (2^{e-1})}}\frac{(2^e-1)!!}{a\cdot b}\quad\text{and}\quad
 T_{2,e}=\sum_{\substack{(a,b)\in D_e\\ a\equiv b\ (2^{e-1})}}\frac{(2^e-1)!!}{a\cdot b}.$$
Each summand of $T_{2,e}$ corresponds to a unique element of $S_{e-1}$, and so, mod $2^{e-1}$,
$$T_{2,e}\equiv\sum_{a\in S_{e-1}}\frac{((2^{e-1}-1)!!)^2}{a^2}\equiv2^{e-2}\pmod{2^{e-1}}$$
by Proposition \ref{sqprop}.

We will prove $T_{1,e}\equiv0$ mod $2^{e-1}$ by induction. It is true when $e=3$ as we obtain four summands, each with denominator 3. Assume validity for $e-1$.
Every element of $D_{e-1}$ corresponds to four  summands of $T_{1,e}$ which are equal mod $2^{e-1}$.
 We obtain, mod $2^{e-1}$,
 $$T_{1,e}\equiv 4\sum_{(a,b)\in D_{e-1}}\frac{((2^{e-1}-1)!!)^2}{a\cdot b}=4(2^{e-1}-1)!!(T_{1,e-1}+T_{2,e-1})\equiv0\quad(2^{e-1}),$$
 using the induction hypothesis for $4T_{1,e-1}$ and the already-proved result for $4T_{2,e-1}$.
\end{proof}
\def\line{\rule{.6in}{.6pt}}

\end{document}